\documentclass[11pt]{amsart}
\usepackage[english]{babel}
\usepackage{amssymb,graphicx,latexsym,xcolor}
\usepackage{hyperref}   
\usepackage{mathtools}  
\mathtoolsset{showonlyrefs}

\usepackage[space]{grffile}

\newcommand{\assign}{:=}
\newcommand{\tmem}[1]{{\em #1\/}}
\newcommand{\tmop}[1]{\ensuremath{\operatorname{#1}}}
\theoremstyle{plain}
\newtheorem{definition}{Definition}
\theoremstyle{plain}
\newtheorem{lemma}{Lemma}
\theoremstyle{plain}
\newtheorem{proposition}{Proposition}
\theoremstyle{remark}
\newtheorem{remark}{Remark}
\theoremstyle{plain}
\newtheorem{theorem}{Theorem}
\usepackage{mathtools}
\mathtoolsset{showonlyrefs}

\begin{document}
\title{Positive mass theorem and free boundary minimal surfaces}
\author{Xiaoxiang Chai}
\date{Nov. 15, 2018}
\begin{abstract}
  Built on a recent work of Almaraz, Barbosa, de Lima on positive mass theorems on
  asymptotically flat manifods with a noncompact boundary, we apply free boundary
  minimal surface techniques to prove their positive mass theorem and study the existence of positive scalar curvature
  metrics with mean convex boundary on a connected sum of the form $(\mathbb{T}^{n-1}\times[0,1])\#M_0$.
\end{abstract}
\email{ChaiXiaoxiang@gmail.com}
\maketitle
\section{Introduction}
An asymptotically flat manifold is used to model an isolated gravitational
system in physics. The positive mass conjecture states that if the system has
nonnegative local mass density, then the system must have nonnegative total
mass measuring at spatial infinity. Schoen and Yau {\cite{schoen1979}} in
1970s established the positive mass theorem for time-symmetric case of the
conjecture using minimal surfaces. They proved the three dimensional case. It
is also called the Riemannian positive mass theorem. When the dimension is
less than eight, the positive mass theorem can be reduced down to dimension
three, see {\cite{schoen-variational-1989}}. Witten {\cite{witten-new-1981}}
found an elegant proof for the non-time-symmetric case (spacetime version)
using spinor techniques and a mathematically rigorous account of his proof can
be found in {\cite{parker-wittens-1982}}. Witten's proof applies for spin
manifolds of all dimensions. As for spacetime version without spin assumption,
there is a recent work of Eichmair, Huang, Lee and Schoen
{\cite{eichmair-spacetime-2015}} which uses marginally outer trapped surface
(abbr. MOTS) to replace minimal surfaces in the argument.

There is also a lot of work to extend positive mass theorem to hyperbolic
settings. {\cite{wang-mass-2001}} and {\cite{chrusciel-mass-2003}} use
Witten-type arguments to prove a positive mass theorem for asymptotically
hyperbolic manifolds. Later Andersson, Cai, and Galloway \
{\cite{andersson-rigidity-2008}} uses the BPS brane action to give a proof of
the non-spin case.

We first recall the definition of a standard asymptotically flat manifold.

\begin{definition}
  \label{AF}(Asymptotically flat) We say that $(M^n, g)$ is asymptotically
  flat with decay rate $\tau > 0$ if there exists a compact subset $K \subset
  M$ and a diffeomorphism $\Psi : M\backslash K \rightarrow \mathbb{R}^n
  \backslash \bar{B}_1 (0)$ such that the following asymptotics holds as $r
  \rightarrow + \infty$:
  \[ |g_{i j} (x) - \delta_{i j} | + r |g_{i j, k} | + r^2 |g_{i j, k l} | = o
     (r^{- \tau}) \]
  where $\tau > \frac{n - 2}{2}$.
\end{definition}

\begin{definition}
  (ADM mass, after Arnowitt, Deser and Misner
  {\cite{arnowitt-canonical-1960}}) Let $(M^n, g)$ be the manifold specified
  above, assume that the scalar curvature $R_g$ of $M$ is integrable, then the
  quantity defined for an asymptotically flat manifold $M$ below
  \[ E_{\ensuremath{\operatorname{ADM}}} = \lim_{r \to + \infty} \{ \int_{S_r}
     (g_{ij, j} - g_{jj, i}) \mu^i dS_r \} \]
  is called the ADM mass. Here, $x = (x_1, \cdots, x_n)$ is the coordinate
  system induced by $\Psi$, $r = |x|$, $g_{i j}$ are the components of $g$
  with respect to $x$ and the comma denotes partial differentiation. $S_r$
  denotes the standard sphere of radius $r$, $\mu^i$ normal to $S_r$ under
  Euclidean metric and the comma denotes partial differentiation.
\end{definition}

The definition of the ADM mass relies on the choice of coordinates and its
geometric invariance of ADM mass is proved by Bartnik
{\cite{bartnik-mass-1986}}. Then the positive mass theorem says that

\begin{theorem}
  (Positive mass theorem {\cite{schoen1979,schoen-positive-2017}}) If $(M, g)$
  is asymptotically flat with scalar curvature $R_g \geqslant 0$ and $R_g$ is
  integrable, then $E_{\ensuremath{\operatorname{ADM}}} \geqslant 0$.
  $E_{\ensuremath{\operatorname{ADM}}} = 0$ if and only if the manifold $(M,
  g)$ is isometric to the standard Euclidean space $(\mathbb{R}^n, \delta)$.
\end{theorem}

The seminal work {\cite{schoen1979}} reveals deep connections between the
geometry of non-compact minimal surface in asymptotically flat 3-manifolds and
non-negative scalar curvature. We briefly outline their ideas below.

Assume the ADM mass $E_{\ensuremath{\operatorname{ADM}}} < 0$, then the metric
can be perturbed into a metric which is conformally flat at infinity (i.e.
outside a compact set) such that the scalar curvature $R_g$ has a strict
positive sign and negativity of the mass is preserved. Then choose a large
number $\sigma > 0$, let
\[ \Gamma_{\sigma, a} = \{(\hat{x}, x_n) \in \mathbb{R}^n : | \hat{x} | =
   \sigma, x_n = a\}, \]
we find a minimal hypersurface $\Sigma_{\sigma, a}$ solving the Plateau
problem in $M$ with boundary $\Gamma_{\sigma, a}$. The deformed metric allows
fixing two coordinate slabs $\{x_n = \pm \Lambda\}$ such that any
$\Sigma_{\sigma, a}$ realizing the minimum among all $\{| \Sigma_{\sigma, a}
|\}_{a \in [- a_0, a_0]}$ lies strictly between the slabs. Moreover, it is
possible to choose a number $a = a (\sigma) \in (- a_0, a_0)$ such that
$\Sigma_{\sigma} = \Sigma_{\sigma, a (\sigma)}$ has the least area among all
$\{\Sigma_{\sigma, a} \}_{a \in [- a_0, a_0]}$. We can take a subsequence
$\sigma_i \to \infty$ such that $\Sigma_{\sigma}$ converge to a strongly
stable minimal hypersurface. The contradiction will follow from the strong
stability and Gauss-Bonnet theorem.

The technical part of their proof is to handle asymptotics. Lohkamp
{\cite{lohkamp-scalar-1999}} observed that if the mass is negative, the metric
can be transformed further into a metric which is Euclidean at infinity. This
allows compactification by identifying edges of a large cube. The compactified
manifold $M'$ has nonnegative scalar curvature and is not flat. By argue that
the $(n - 1)$-th homology group of $M'$ is non-trivial and that there exists
an area-minimizing hence stable minimal hypersurface in $M'$. The stability of
this hypersurface allows a dimension reduction argument from dimension seven
down to dimension three, and the proof of dimension three finishes with the
Gauss-Bonnet theorem.

Recently Almaraz, Barbosa and de Lima in {\cite{almaraz-positive-2016}}
introduce a notion of an asymptotically flat manifold with a non-compact
boundary as well as an ADM mass:

\begin{definition}
  \label{AFb}(Asymptotically flat with a noncompact boundary) We say that $(M,
  g)$ is asymptotically flat with decay rate $\tau > 0$ if there exists a
  compact subset $K \subset M$ and a diffeomorphism $\Psi : M\backslash K
  \rightarrow \mathbb{R}_+^n \backslash \bar{B}_1^+ (0)$ such that the
  following asymptotics holds as $r \rightarrow + \infty$:
  \[ |g_{i j} (x) - \delta_{i j} | + r |g_{i j, k} | + r^2 |g_{i j, k l} | = o
     (r^{- \tau}) \]
  where $\tau > \frac{n - 2}{2}$.
\end{definition}

\begin{definition}
  (ADM mass with a noncompact boundary) {\itshape{ADM mass}} for $M$ is given
  by
  \begin{equation}
    m_{\ensuremath{\operatorname{ADM}}} = \lim_{r \to + \infty} \left\{
    \int_{S_{r, +}^{n - 1}} (g_{ij, j} - g_{jj, i}) \mu^i dS_{r, +}^{n - 1} +
    \int_{S_r^{n - 2}} g_{1 n} \vartheta^1 dS_r^{n - 2} \right\}
    \label{energy-def}
  \end{equation}
  where \ $\mathbb{R}_+^n = \{x \in \mathbb{R} : x_1 \geqslant 0\}$ and
  $\bar{B}_1^+ (0) = \{x \in \mathbb{R}_+^n : |x| \leqslant 1\}$. We also use
  the Einstein summation convention with index ranges $i, j, k = 1, \cdots, n$
  and $a, b, c = 2, \cdots, n$. Observe that along $\partial M$, $\{\partial_a
  \}$ spans $T \partial M$ while $\partial_1$ points inwards. $S_{r, +}^{n -
  1} \subset M$ is a large coordinate hemisphere of radius $r$ with outward
  unit normal $\mu$, and $\vartheta$ is the outward pointing unit co-normal to
  $S_r^{n - 2} = \partial S_{r, +}^{n - 1}$, viewed as the boundary of the
  bounded region $\Sigma_r \subset \Sigma$ .
\end{definition}

We write $m (M, g)$ if we want to emphasize the dependence on the manifold and
the metric, and we write $m_g$ for short if the manifold $M$ is clear from the
context. See Fig. \ref{fb-af} for a hemisphere in such an asymptotically flat
manifold.

\begin{figure}[h]
  \includegraphics{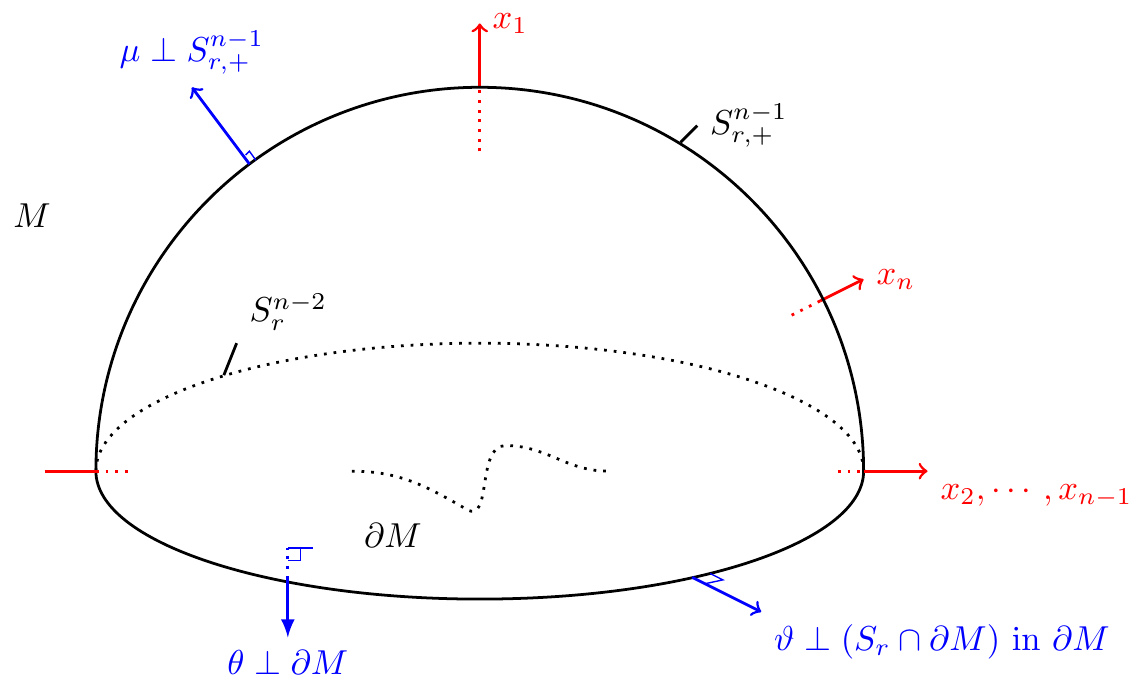}
  \caption{\label{fb-af}A hemisphere of radius $r$ in an asymptotically
  flat manifold with a non-compact boundary.}
\end{figure}

Motivated by the proof by Schoen and Yau {\cite{schoen1979}} using minimal
hypersurface techniques, \ S. Almaraz, E. Barbosa and L. de Lima proved a
positive mass theorem for asymptotically flat manifolds with a non-compact
boundary, more specifically,

\begin{theorem}
  \label{positive mass theorem}{\cite[Theorem 1.3]{almaraz-positive-2016}}
  When $3 \leqslant n \leqslant 7$ and if $(M, g)$ is asymptotically flat $R_g
  \geqslant 0$ and $H_g \geqslant 0$ then $m_{\ensuremath{\operatorname{ADM}}}
  \geqslant 0$, with the equality occurring if and only if $(M, g)$ is
  isometric to $(\mathbb{R}^n_+, \delta)$.
\end{theorem}

Their method is to perturb the metric, making the manifold $M$ conformally
flat at infinity and the mean curvature of $\partial M$ strictly positive,
therefore $\partial M$ serves a barrier for the area-minimizing hypersurface
to exist.

We provide a different approach which uses free boundary minimal hypersurfaces
instead.

We also study the geometry of $(\mathbb{T}^{n - 1} \times [0, 1])
\#M_0$ where $M_0$ is non-flat. We prove the following

\begin{theorem}
  \label{thm:non-existence}There does NOT exist a metric on $(\mathbb{T}^{n -
  1} \times [0, 1]) \#M_0$ with nonnegative scalar curvature and nonnegative
  mean curvature along the boundary where $3 \leqslant n \leqslant 7$.
\end{theorem}

In fact, the only metric with nonnegative scalar curvature and nonnegative
mean curvature is flat which in turn will force $M_0$ to be flat. The
existence of positive scalar curvature metrics on $\mathbb{T}^n \#M_0$ were
studied most notably by works of Schoen, Yau {\cite{schoen-structure-1979}}
and Gromov, Lawson {\cite{gromov-positive-1983}}.
This non-existence result Theorem \ref{thm:non-existence} was essentially due to Gromov and Lawson {\cite{gromov-positive-1983}}. For the convenience of the reader, we include our sketch of their proof. Note that their proof used minimal surfaces techniques and we use instead free boundary minimal surfaces. 

Also, other results of the
latter require a spin assumption on the manifold $\mathbb{T}^n \#M_0$.
For Theorem \ref{thm:non-existence}, the corresponding
spin versions can be established via the analysis of Dirac spinors with
integrated Bochner formula (see Hijazi, Motiel and Zhang {\cite[Eq.
(2.3)]{hijazi-eigenvalues-2001}}). We focus here on the non-spin case, and the
dimension assumption $3 \leqslant n \leqslant 7$ is a technical one.

Recently, Schoen and Yau {\cite{schoen-positive-2017}} develop a minimal
slicing theory and use it to settle the non-spin higher dimensional positive
mass theorem and also established a non-existence result of positive scalar
curvature metrics on $\mathbb{T}^n \#M_0$ with $n$ greater than seven. We
intend to generalize the theory to the boundary setting in future works.

The article is organized as follows:

In Section \ref{fbmh}, we record some basics of free boundary minimal
hypersurfaces and that of conformal changes. In Section \ref{schoen approach},
we present another approach which resembles more formally with Schoen and
Yau's original approach by replacing minimal hypersurface with free boundary
minimal hypersurface instead. We follow closely to Schoen's article
{\cite{schoen-variational-1989}}. \ In Section \ref{geometry}, we study the
relationship of the geometry of $(\mathbb{T}^{n - 1} \times [0, 1]) \#M_0$ and
the positive mass theorem. In the appendix, we give details of some of the
computations.

{\bfseries{Acknowledgements.}} This work is part of the author's PhD thesis at the Chinese
University of HK. He would like to thank his PhD advisor Prof. Martin Man-chun Li for suggesting this problem
and many helpful discussions, and also for continuous encouragement and support.

\section{free boundary minimal hypersurfaces and conformal
geometry}\label{fbmh}

In this section, $(M^n, g)$ is a smooth manifold of dimension $n$ with
nonempty boundary $\partial M$ and $\Sigma^{n - 1}$ is an immersed
hypersurface whose boundary lies in $\partial M$. Let $\Sigma$ be a free
boundary minimal hypersurface which is a critical point of the volume
functional $V (\Sigma)$ among all surfaces whose boundary lies in $\partial
M$. We compute the first and second variation of the functional $V (\Sigma)$
with respect to variational vector field $X$ whose restriction on $\partial
\Sigma$ is tangential to $\partial M$. This computation also fills some
calculational details missing in {\cite{schoen-variational-1989}}.

We adopt the following notations.

{\bfseries{Notations.}} We use $(\Sigma, \partial \Sigma) \looparrowright (M,
\partial M)$ to denote that $\Sigma$ is a hypersurface in $M$ with boundary
lying on $\partial M$. $B (X, Y) := \langle \nabla_X \nu, Y \rangle$ is the
second fundamental form of $\Sigma$ in $M$ where $\nu$ is a fixed unit normal
to $\Sigma$ in $M$ and $X, Y$ are tangent to $\Sigma$ when restricted to
$\Sigma$. Let $A (Y, Z) := \langle \nabla_Y \eta, Z \rangle$ where $\eta$ is
the outward normal of $\partial M$ in $M$ and $Y, Z$ are tangent to $\partial
M$ in $M$. When $\eta$ is also normal to $\partial \Sigma$ in $\Sigma$, the
second fundamental form $A$ evaluated on $T \partial \Sigma \times T \partial
\Sigma$ can be expressed as the second fundamental form of $\partial \Sigma$
in $\Sigma$ as well.

\subsection{Basics of free boundary minimal hypersurfaces}

We have the following definition

\begin{definition}
  (Free boundary minimal hypersurface) $(\Sigma, \partial \Sigma)
  \looparrowright (M, \partial M)$ is said to be a free boundary minimal
  hypersurface if the first variation of the volume functional $V (\Sigma)$
  vanishes along any vector field $X$ which only has components tangential to
  $\partial M$ along $\partial \Sigma$.
\end{definition}

It is well known that the first variation of any hypersurface $\Sigma$ is
given by
\[ \delta \Sigma (X) := \delta V (X) |_{\Sigma} = \int_{\Sigma}
   \mathrm{div}_{\Sigma} X = \int_{\Sigma} H \langle X, \nu \rangle +
   \int_{\partial \Sigma} \langle X, \eta \rangle, \]
here $\nu$ is a fixed normal of $\Sigma$, $H = \mathrm{div}_{\Sigma} \nu$ is
the mean curvature of $\Sigma$ and $\eta$ is the outward normal of $\partial
\Sigma$ in $\Sigma$. We see immediately from the definition of a free boundary
minimal hypersurface, $\Sigma$ is free boundary minimal if and only if
\[ H \equiv 0 \text{ on } \Sigma \text{ and } \eta \perp \partial M \text{
   along } \partial \Sigma . \]
We record the second variation of a free boundary minimal hypersurface in the
following theorem and postpone the calculation to the appendix.

\begin{theorem}
  \label{2nd variation}Given a free boundary minimal hypersurface $\Sigma$ in
  $M$, let $\nu$ be a normal to $\Sigma$ in $M$ and $X$ be a variational
  vector field. $X$ admits the decomposition that $X = \varphi \nu + T$ where
  $T$ is tangent to $\Sigma$ and $\nu$ is normal to $\Sigma$. Let the normal
  component of $\nabla_X X$ be $\phi \nu$ and tangent component be $\hat{Z}$.
  Assume that $X$ is tangent to $\partial M$ along $\partial \Sigma$ then the
  second variation of volume is
  \[ \delta^2 \Sigma (X) := \delta^2 V (X) |_{\Sigma} = \int_{\Sigma} F_X d
     \mathrm{vol}_{\Sigma} \]
  where the density is given by
  \begin{align}
    F_X = & - \varphi^2 \mathrm{Ric} (\nu) - \varphi^2 |B|^2 + | \nabla
    \varphi |^2\\
    & + \mathrm{div} (T \mathrm{div} T - \nabla_T T) + \mathrm{div} \hat{Z} -
    2 (\varphi B_{ij} T_i)_{; j} .
  \end{align}
\end{theorem}

If the variational vector field $X$ is normal to $\Sigma$, then we can write
the second variation in a simpler form
\[ \delta^2 \Sigma (f \nu) = \int_{\Sigma} | \nabla f|^2 - (\mathrm{Ric} (\nu)
   + |B|^2) f^2 - \int_{\partial \Sigma} f^2 A (\nu, \nu) \]
where $f \in C^{\infty}_c (\Sigma)$.

\begin{definition}
  We say that $\Sigma$ is a stable free boundary minimal surface if for any $f
  \in C^{\infty}_c (\Sigma)$, the second variation $\delta^2 \Sigma (f \nu)
  \geqslant 0$. The inequality $\delta^2 \Sigma (f \nu) \geqslant 0$ is called
  the stability inequality.
\end{definition}

We write down the stability inequality for free boundary minimal hypersurfaces
in full,
\begin{equation}
  \delta^2 \Sigma (f \nu) = \int_{\Sigma} (- f^2 \mathrm{Ric} (\nu) - f^2
  |B|^2 + | \nabla f|^2) - \int_{\partial \Sigma} f^2 A (\nu, \nu) \geqslant
  0. \label{stability}
\end{equation}
With the Gauss-Codazzi equation,
\[ \mathrm{Ric} (\nu) + |B|^2 = \frac{1}{2} R_M - \frac{1}{2} R_{\Sigma} +
   \frac{1}{2} |B|^2 . \]
This is the fundamental observation made by Schoen and Yau
{\cite{schoen1979}}. Decomposition along $\partial \Sigma$ of the mean
curvature $H_{\partial M}$ gives that
\begin{equation}
  H_{\partial M} = \Sigma_{j = 1}^{n - 2}  (\nabla_{e_j} \eta, e_j) + A (\nu,
  \nu) = H_{\partial \Sigma} + A (\nu, \nu) \label{decomposition of mean
  curvature}
\end{equation}
where the orthonormal frame $e_j$ is tangent to $\partial \Sigma$, $e_{n - 1}
= \eta$ and $e_n = \nu$. We insert these identities back to
{\eqref{stability}} and get
\begin{equation}
  \int_{\Sigma} \left[ | \nabla f|^2 - f^2 (\frac{1}{2} R_M - \frac{1}{2}
  R_{\Sigma} + \frac{1}{2} |B|^2) \right] - \int_{\partial \Sigma} f^2
  (H_{\partial M} - H_{\partial \Sigma}) \geqslant 0 \label{stability2} .
\end{equation}
A rewrite of this inequality assuming that $R_M > 0$ and $H_{\partial M}
\geqslant 0$,
\begin{equation}
  \begin{array}{ll}
    \int_{\Sigma} | \nabla f|^2 + \frac{1}{2} f^2 R_{\Sigma} + \int_{\partial
    \Sigma} f^2 H_{\partial \Sigma} & > 0
  \end{array} \text{for all } 0 \neq f \in C^{\infty} (\Sigma)
  \label{stability-use}
\end{equation}
using $R_M$ is strictly positive everywhere is sufficient for our purpose.

\subsection{Conformal changes}

Given any manifold $(M^n, g)$, take any $u > 0$ on $M$, the conformal changed
metric $\hat{g} = u^{\frac{4}{n - 2}} g$ gives a law for the change of the
scalar curvature and the mean curvature of the boundary.

Denote $c_n = \frac{n - 2}{4 (n - 1)}$. We define the {\itshape{conformal
Laplacian}} by
\[ L = c_n R_g - \Delta_g u \quad \text{ in } M, \]
and the scalar curvature under $\hat{g}$ is given by
\[ R_{\hat{g}} = c_n^{- 1} u^{- \frac{n + 2}{n - 2}} (c_n R_g - \Delta_g u) =
   c_n^{- 1} u^{- \frac{n + 2}{n - 2}} L u. \]
We define also an operator acting along the boundary
\[ B = \partial_{\nu} + 2 c_n H_g \quad \text{on } \partial M, \]
and the mean curvature $H_{\hat{g}}$ of the boundary $\partial M$ under
$\hat{g}$ is given by
\[ H_{\tilde{g}} = \frac{1}{2} c_n^{- 1} u^{- \frac{2}{n - 2}} (2 c_n H_g +
   \partial_{\nu} u) = \frac{1}{2} c_n^{- 1} u^{- \frac{2}{n - 2}} B u \quad
   \text{along } \partial M. \]
where $\partial_{\nu}$ denotes the derivative along the outward unit normal
$\nu$ to $\partial M$ in $M$. We write often $L (M, g) = L$ and $B (M, g) = B$
to avoid confusion.

We derive a simple consequence of {\eqref{stability-use}} encoded in the
following lemma,

\begin{lemma}
  \label{lm:positive minimal}Given any compact manifold $(M^n, g)$ with
  boundary $\partial M$, suppose that
  \[ \int_M | \nabla f|^2 + \frac{1}{2} R_M f^2 + \int_{\partial M} f^2
     H_{\partial M} > 0 \]
  for all $0 \neq f \in C^{\infty} (M)$. Then $M$ admits a positive scalar
  curvature metric $\hat{g}$, and under this metric the boundary $\partial M$
  is minimal.
\end{lemma}

\begin{proof}
  The eigenvalue problem
  \[ \left\{\begin{array}{ll}
       L u = \lambda u & \quad \text{in } M,\\
       B u = 0 & \quad \text{ on } \partial M
     \end{array}\right. \]
  admits a positive solution $u > 0$.

  Using Rayleigh quotient and that $2 c_n < 1$,
  \begin{align}
    \lambda & = \frac{\int_M (| \nabla u |^2 + c_n R_M u^2) + 2 c_n
    \int_{\partial M} H_{\partial M} u^2}{\int_M u^2}\\
    & \geqslant 2 c_n \left[ \frac{\int_M \left( | \nabla u |^2 + \frac{1}{2}
    R_M u^2 \right) + \int_{\partial M} H_{\partial M} u^2}{\int_M u^2}
    \right] > 0.
  \end{align}
  Let $\tilde{g} = u^{\frac{4}{n - 2}} g$, then
  \begin{align*}
    R_{\tilde{g}} & = c_n^{- 1} u^{- \frac{n + 2}{n - 2}} (c_n R_g - \Delta_g
    u)\\
    & = c_n^{- 1} u^{- \frac{n + 2}{n - 2}} \lambda u = \lambda c_n^{- 1}
    u^{- \frac{4}{n - 2}} > 0
  \end{align*}
  and
  \[ H_{\tilde{g}} = u^{- \frac{2}{n - 2}} (H_g + \frac{1}{2} c_n^{- 1}
     \partial_{\nu} u) = 0 \text{ along } \partial M \]
  give that the metric $\tilde{g}$ is the desired metric.
\end{proof}

Similarly, we have

\begin{lemma}
  \label{lm:flat minimal}Given any compact manifold $(M^n, g)$ with boundary
  $\partial M$, suppose that
  \[ \int_M | \nabla f|^2 + \frac{1}{2} R_M f^2 + \int_{\partial M} f^2
     H_{\partial M} > 0 \]
  for all $0 \neq f \in C^{\infty} (M)$. Then $M$ admits a scalar-flat metric
  $\hat{g}$, under this metric the boundary $\partial M$ is strictly mean
  convex.
\end{lemma}

\begin{proof}
  The proof is similar to the previous lemma, except that we consider the
  following Steklov-type eigenvalue problem instead:
  \begin{equation}
    L \phi = 0 \quad \text{in } M, \hspace{2.0em} B \phi = \lambda \phi \quad
    \text{on } \partial M
  \end{equation}
  We omit the details.
\end{proof}

\section{An alternate proof of Theorem \ref{positive mass
theorem}}\label{schoen approach}

In this section, we provide another proof of the positive mass theorem
(Theorem \ref{positive mass theorem}) using free boundary minimal
hypersurfaces.

\subsection{Step 1: Existence of area-minimizing hypersurface with free
boundary}

We assume on the contrary that $m_g < 0$. By the density theorem
{\cite[Proposition 4.1]{almaraz-positive-2016}}, we can assume that $g =
h^{\frac{4}{n - 2}} \delta$, $R_g > 0$ on $M$ and $H_g > 0$ on $\partial M$
where $h (x) = 1 + C (n) m_g | x |^{2 - n} + O (| x |^{1 - n})$, where $C (n)$
is a constant depending only on the dimension. Consider the vector field $\eta
= h^{- \frac{2}{n - 2}} \partial_n$. We compute the divergence of $\eta$ with
respect to $g$
\[ \mathrm{\ensuremath{\operatorname{div}}}_g \eta = - 2 (n - 1) C (n) m_g
   \frac{x^n}{| x |^n} + O (| x |^{- n}) \]
In particular we see that $\mathrm{\ensuremath{\operatorname{div}}_g \eta > 0}
$ for $x^n \geqslant a_0$ and $\mathrm{\ensuremath{\operatorname{div}}_g (-
\eta)} > 0$ for $x^n \leqslant - a_0$ for some constant $a_0$. Let $\sigma$ be
a large real number, let
\[ \Gamma_{\sigma, a} = \{ x = (\bar{x}, x^n) : x^n = a, | \bar{x} | = \sigma,
   x^1 \geqslant 0 \} \]
and
\[ C_{\sigma} = \left\{ x = (\bar{x}, x^n) : | \bar{x} | = \sigma, x^1
   \geqslant 0 \text{ or } x^1 = 0 | \bar{x} | \leqslant \sigma \right\} . \]
The half cylinder $C_{\sigma}$ with $\partial M$ bounds to the interior a
region $\Omega_{\sigma}$ in $M$. We solve a Plateau problem within the class
of hypersurfaces with partially free boundary on $\partial M$ and fixed
boundary $\Gamma_{\sigma, a}$ and obtain an $(n - 1)$-dimensional
hypersurfaces $\Sigma_{\sigma, a}$ with least area among all such
hypersurfaces. Where the free boundary and fixed boundary meet is called the
{\itshape{corner of the hypersurface}}. In our situation, the corner is the
following set
\[ \Lambda_{\sigma, a} = \{ x = (\bar{x}, x^n) : | \bar{x} | = \sigma, x^1 =
   0, x^n = a \} . \]
When solving a Plateau problem, regularity issues will often arise. The
requirement that the dimension $3 \leqslant n < 8$ is one of them (see
{\cite{federer-geometric-1996}}), and in particular regularity is a problem at
the corner. We now list standard known facts about the regularity of
$\Sigma_{\sigma, a}$.

The interior regularity of $\Sigma_{\sigma, a}$ is just classical geometric
measure theory (see {\cite{federer-geometric-1996}}). The regularity at the
free boundary of the boundary away from the corner $\Lambda_{\sigma, a}$ is
shown by Gruter {\cite{gruter-optimal-1987,gruter-regularity-2009}} and the
regularity near $\Gamma_{\sigma, a} \sim \Lambda_{\sigma, a}$ follows from the
work of Hardt and Simon {\cite{hardt-boundary-1979}}. Although Gruter
{\cite{gruter-remark-1990}} claimed some regularity results at the corner, but
we have not seen those get published.

In conclusion, regularity can only be an issue at the corners. However, we
are able to bypass this when taking limits.

\begin{figure}[h]
  \includegraphics{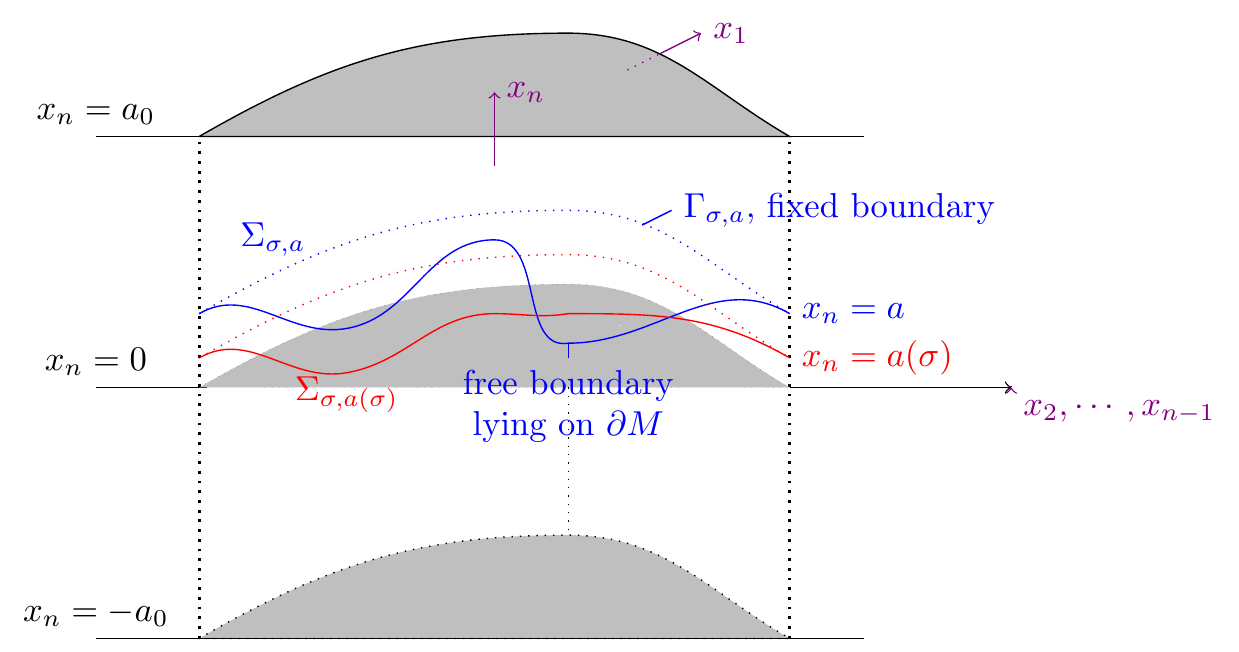}
  \caption{\label{fb-choice}Solving a Plateau problem with a partially
  free boundary lying on $\partial M$ and fixed boundary $\Gamma_{\sigma,
  a}$.}
\end{figure}
For any large $\sigma > 0$, let
\[ V (\sigma) = \min \{ \ensuremath{\operatorname{Vol}} (\Sigma_{\sigma, a}) :
   a \in [- a_0, a_0] \} . \]
First, we note that

{\bfseries{Claim 1:}} The function $a \mapsto \mathcal{H}^{n - 1}
(\Sigma_{\sigma, a})$ is continuous.

To prove this, take any two numbers $a_1, a_2 \in [- a_0, a_0]$, we can
further assume that $a_1 < a_2$. $\Sigma_{\sigma, a_1}$ is area-minimizing
with free boundary on $\partial M$ and fixed boundary $\Gamma_{\sigma, a'}$.
The union of $\Sigma_{\sigma, a_2}$ and $\cup_{a \in [a_1, a_2]}
\Gamma_{\sigma, a}$ serves as a comparison surface (which is only piecewise
smooth) for $\Sigma_{\sigma, a_1}$. By volume comparison,
\[ \mathcal{H}^{n - 1} (\Sigma_{\sigma, a_1}) \leqslant \mathcal{H}^{n - 1}
   (\Sigma_{\sigma, a_2}) +\mathcal{H}^{n - 1} (\cup_{a \in [a_1, a_2]}
   \Gamma_{\sigma, a}), \]
where the last term will be bounded by $C (a_2 - a_1) \sigma^{n - 2}$ and the
constant $C$ depend only on the dimension $n$ and the decay properties of the
metric $g$. Switching roles of $\Sigma_{\sigma, a_1}$ and $\Sigma_{\sigma,
a_2}$, we get
\[ |\mathcal{H}^{n - 1} (\Sigma_{\sigma, a_1}) -\mathcal{H}^{n - 1}
   (\Sigma_{\sigma, a_2}) | \leqslant C |a_2 - a_1 | \sigma^{n - 2}, \]
and hence finish the proof of the {\bfseries{Claim 1}}.

Hence, the $V (\sigma)$ is attained by some hypersurface $\Sigma_{\sigma, a}$
where $a = a (\sigma) \in [- a_0, a_0]$. Also, we claim further that

{\bfseries{Claim 2}}. There exists $a = a (\sigma) \in (- a_0, a_0)$ such that
$\ensuremath{\operatorname{Vol}} (\Sigma_{\sigma, a}) = V (\sigma)$.

To show that $a (\sigma) < a_0$, let $\Omega_{\sigma, a}$ be the region lying
below $\Sigma_{\sigma, a}$ and
\[ U_{\sigma, a} = \{ (\bar{x}, x^n) \in \Omega_{\sigma, a} : x^n > a_0 -
   \delta \} \]
where $\delta$ is chosen small such that
$\mathrm{\ensuremath{\operatorname{div}}_g \eta > 0} $ for $x^n > a_0 -
\delta$. We show by contradiction that $U_{\sigma, a} = \emptyset$. We have by
the divergence theorem,
\begin{eqnarray*}
  0 < \int_{U_{\sigma, a}} \ensuremath{\operatorname{div}}_g \eta & = &
  \int_{\Sigma_{\sigma, a} \cap \{ x^n \geqslant a_0 - \delta \}} \langle
  \eta, \nu \rangle + \int_{\Omega_{\sigma, a} \cap \{ x^n = a_0 - \delta \}}
  \langle \eta, - \eta \rangle\\
  & = & \int_{\Sigma_{\sigma, a} \cap \{ x^n \geqslant a_0 - \delta \}}
  \langle \eta, \nu \rangle -\ensuremath{\operatorname{Vol}} (\Omega_{\sigma,
  a} \cap \{ x^n = a_0 - \delta \})
\end{eqnarray*}
where on the right hand side integration on other pieces vanish because $\eta$
is tangent to $C_{\sigma}$. By Cauchy-Schwarz inequality we arrive
\[ \ensuremath{\operatorname{Vol}} (\Omega_{\sigma, a} \cap \{ x^n = a_0 -
   \delta \}) <\ensuremath{\operatorname{Vol}} (\Sigma_{\sigma, a} \cap \{ x^n
   \geqslant a_0 - \delta \}) . \]
In the case when $a < a_0 - \delta$, we can project pieces of $\Sigma_{\sigma,
a}$ above $\{x_n = a_0 - \delta\}$ down to $\{ x^n = a_0 - \delta \}$ and get
a smaller area. This contradicts the minimality of $\Sigma_{\sigma, a}$ among
all hypersurfaces with $\Gamma_{\sigma, a}$ as boundary in $\Omega_{\sigma}$.
When $a > a_0 - \delta$ we can do the same projection and we get a surface
with fixed boundary $\Gamma_{\sigma, a_0 - \delta}$ and its area is strictly
less that $V (\sigma)$, and this contradicts the choice of $a$. Therefore
$U_{\sigma, a} = \emptyset$. We obtain similarly the lower bound for $a
(\sigma)$. In conclusion, $a (\sigma) \in (- a_0, a_0)$.

\subsection{Step 2: Strong stability and limiting behavior as $\sigma \to
\infty$}

Let $\Sigma_{\sigma} := \Sigma_{\sigma, a (\sigma)}$ be one of the
hypersurfaces which realizes the minimum volume $V (\sigma)$. Let $X_1$ be a
fixed vector field on $M$ which is equal to $\partial_n$ outside a compact
set. Let $X_0$ be a vector field of compact support, and let $X = X_0 + \alpha
X_1$ where $\alpha$ is a real number. The vector field $X$ generates a
one-parameter group of diffeomorphism $F_t$. $F_t$ gives a variation of
$\Sigma_{\sigma}$ and
\[ \frac{d}{d t} \ensuremath{\operatorname{Vol}} (F_t (\Sigma_{\sigma})) |_{t
   = 0}  = 0, \frac{d^2}{d t^2} \ensuremath{\operatorname{Vol}} (F_t
   (\Sigma_{\sigma})) |_{t = 0}  \geqslant 0 \]

We choose a sequence $\sigma_i \rightarrow \infty$ such that
$\Sigma_{\sigma_i}$ converge to a limiting volume minimizing hypersurface
$\Sigma \subset M$. We see that the possible singularity at the corner goes
away by taking limits.

We claim that {\color{blue}{{\tmem{}}}}$\Sigma$ is a graph of a function $f$
near infinity, that is to say $\Sigma$ outside a compact set is given by $x_n
= f (x_1, \ldots, x_n)$ with $x_1 \geqslant 0$. We prove this claim by scaling
techniques. Denote $\bar{x} = (x_1, \ldots, x_{n - 1})$. Take $p = (\bar{x},
x_n) \in \Sigma$ with $| \bar{x} | = 2 \sigma$ for $\sigma$ sufficiently
large. $\Lambda$ is a number greater than 1, let $S_{\Lambda} = \{x \in M :
x_n = \Lambda\}$. Consider $\Omega_{\sigma}$ which is bounded by $\partial M$,
$C_{\sigma}$ and the slabs $S_{- a_0}$ and $S_{a_0}$. Using volume comparison,
we see that
\[ | \Sigma \cap \Omega_{\sigma} | \leqslant \frac{\omega_{n - 1}}{2}
   \sigma^{n - 1} + O (\sigma^{- 1}) \sigma^{n - 1} . \]
Outside a large enough compact set, $f$ satisfies the minimal surface equation
\begin{equation}
  \sum_{i j} \left( \delta_{i j} - \frac{f_{, i} f_{, j}}{1 + | \partial f
  |^2} \right) f_{, i j} + \frac{2 (n - 1)}{n - 2} \sqrt{1 + | \partial f |^2}
  \frac{\partial}{\partial \nu_0} \log h = 0,
\end{equation}
and on the boundary $\partial \Sigma$ (also outside a compact set)
\begin{equation}
  \partial_1 f = 0.
\end{equation}
Recall that under the metric $g = h^{\frac{4}{n - 2}} \delta$, $h$ is the
conformal factor
\[ h (x) = 1 + C (n) m_{\ensuremath{\operatorname{ADM}}} | x |^{2 - n} + O (|
   x |^{1 - n}) \]
and
\[ \nu_0 = (1 + | \partial f |^2)^{- 1 / 2} (- \partial f, 1) . \]
We record the calculation of decay rate of $f$ in Lemma \ref{decay
calculation}.

Let $D_{\sigma}$ denote the portion of $\Sigma$ bounded by $C_{\sigma}$ to the
interior. According to Lemma \ref{2nd variation}, the density of the second
variation can be written as
\[ F_X = - \varphi^2 \ensuremath{\operatorname{Ric}} (\nu) - \varphi^2 | B |^2
   + | \nabla \varphi |^2 + G \]
where
\[ G =\ensuremath{\operatorname{div}} (T\ensuremath{\operatorname{div}}T -
   \nabla_T T) +\ensuremath{\operatorname{div}} \hat{Z} - 2 (\varphi B_{i j}
   T_i)_{; j} . \]
Here, $T$ is the tangential component of $X$, $\nu$ is the normal of $\Sigma$
in $M$ and $\eta$ is the normal of $\partial D_{\sigma}$ in $\Sigma$.

The integral of $F_X$ over a large bounded region $D_{\sigma}$ is then by
divergence theorem

\begin{align}
  \int_{D_{\sigma}} F_X & = \int_{D_{\sigma}} (- \varphi^2
  \ensuremath{\operatorname{Ric}} (\nu) - \varphi^2 | B |^2 + | \nabla \varphi
  |^2)\\
  & \phantom{i j k} \phantom{i j k} + \int_{\partial D_{\sigma}}
  [(T\ensuremath{\operatorname{div}}T - \nabla_T T, \eta) - 2 \varphi B (T,
  \eta) + (\hat{Z}, \eta)] .
\end{align}
Along the free boundary $\partial \Sigma$ we have

\begin{align}
  (\hat{Z}, \eta) = (\nabla_X X, \eta) & = (\nabla_{T + \varphi \nu} (T +
  \varphi \nu), \eta)\\
  & = (\nabla_T T + \varphi \nabla_T \nu + \varphi^2 \nabla_{\nu} \nu +
  \varphi \nabla_{\nu} T, \eta)\\
  & = (\nabla_T T, \eta) + \varphi B (T, \eta) - \varphi^2 A (\nu, \nu) +
  \varphi (\nabla_{\nu} T, \eta)\\
  & = (\nabla_T T, \eta) + 2 \varphi B (T, \eta) - \varphi^2 A (\nu, \nu)
\end{align}
where $A (Y, Z) := (\nabla_Y \eta, Z)$ for $Y, Z \in T \partial M$ i.e. the
second fundamental form of $\partial M$ in $M$. Since $\langle T, \eta \rangle
= 0$, hence the term left for $\partial \Sigma$ is $- \varphi^2 A (\nu, \nu)$.

From the decay conditions on $h$ and $f$, we see that the boundary terms along
$\partial D_{\sigma} \cap \Sigma$ decay faster than $\sigma^{2 - n}$, then
integral over $\partial D_{\sigma} \cap \Sigma$ tends to zero as $\sigma \to
\infty$. See Lemma \ref{decay calculation} in the appendix. By letting $\sigma
\rightarrow \infty$, we arrive the {\itshape{strong stability inequality}}
\begin{equation}
  \int_{\Sigma} (- \varphi^2 \ensuremath{\operatorname{Ric}} (\nu) - \varphi^2
  | B |^2 + | \nabla \varphi |^2) - \int_{\partial \Sigma} \varphi^2 A (\nu,
  \nu) \geqslant 0 \label{strong stability}
\end{equation}
Applying the same trick as in deriving {\eqref{stability-use}}, we see the
inequality {\eqref{strong stability}} becomes

\begin{align}
  & \int_{\Sigma} | \nabla \varphi |^2 + \frac{1}{2} \varphi^2 R_{\Sigma} +
  \int_{\partial \Sigma} \varphi^2 H_{\partial \Sigma}\\
  \geqslant & \int_{\Sigma} \frac{1}{2} \varphi^2 (R_M + |B|^2) +
  \int_{\partial \Sigma} \varphi^2 H_{\partial M} > 0. \label{conformal-rough}
\end{align}

The condition on $\varphi$ can be derived from the condition on $X$, since
\[ \varphi = \alpha \langle X, \nu \rangle = \alpha \langle \partial_n, \nu
   \rangle = \alpha h^{\frac{2}{n - 2}} (1 + | \partial f |^2)^{- 1 / 2} \]
outside a compact set for a constant $\alpha$.

Since $\varphi - \alpha = O (|x' |^{3 - n})$ we see that $\varphi - \alpha$
has finite mass and therefore we can take $\varphi$ to be any function for
which $\varphi - \alpha$ has compact support or finite mass for some constant
$\alpha$.

\subsection{Step 3: Strong stability and Gauss-Bonnet theorem}

We use the above to obtain a contradiction when $n = 3$. Specifically, taking
$\varphi \equiv 1$ in the stability inequality (\ref{stability}), we have
\[ \int_{\Sigma} K + \int_{\partial \Sigma} k_g = \int_{\Sigma} \frac{1}{2}
   R_{\Sigma} + \int_{\partial \Sigma} H_{\partial \Sigma} > 0 \]
We use the large region $D_{\sigma}$ to approximate the integral, by
Gauss-Bonnet theorem, we have
\[ \int_{\Sigma} K + \int_{\partial \Sigma \cap D_{\sigma}} k_g = 2 \pi \chi
   (D_{\sigma}) - 2 \pi - \int_{\partial D_{\sigma} - \partial \Sigma \cap
   D_{\sigma}} k_g + \alpha_1 + \alpha_2 \label{gauss bonnet with angle} \]
where $\alpha_i$ are the inner angles. Note that $D_{\sigma}$ has at least one
boundary components and possibly has positive genus, hence $\chi (D_{\sigma})
= 2 - 2 g - b \leqslant 2 - 2 \cdot 0 - 1 = 1$. So the right hand side
converge a number less than to zero, yet the left hand side converges to a
positive number by stability {\eqref{conformal-rough}}.

We see a contradiction. Therefore, when $n = 3$,
$m_{\ensuremath{\operatorname{ADM}}} \geqslant 0$.

\subsection{Step 4: Dimension reduction argument}

When $4 \leqslant n \leqslant 7$, we have the induced metric $\hat{g}$ on
$\Sigma$ in terms of coordinates $x^1, \ldots, x^{n - 1}$
\[ \hat{g}_{i j} = h (\bar{x}, f (\bar{x}))^{\frac{4}{n - 2}} (\delta_{i j} +
   f_{, i} f_{, j}) = \delta_{i j} + O (| x |^{2 - n}) \]
where $i, j$ ranges from $1$ to $n - 1$. Therefore, $(\Sigma, \hat{g})$ is
asymptotically flat and has mass zero. We consider the pair of operators
$(L_{\hat{g}}, B_{\hat{g}})$ defined by $L_{\hat{g}} = - \Delta_{\hat{g}} +
c_{n - 1} R_{\hat{g}}^{\Sigma}$ in $\Sigma$ and $B_{\hat{g}} = \partial_{\eta}
+ 2 c_{n - 1} H_{\hat{g}}^{\partial \Sigma}$ on $\partial \Sigma$ where $\eta$
is the outward normal of $\partial \Sigma$ in $\Sigma$ under the metric
$\hat{g}$.

Now we want to find a positive solution $u$ to the boundary value problem
\begin{equation}
  \label{neumann} \left\{ \begin{array}{lcll}
    L_{\hat{g}} u & = & 0 & \text{in } \Sigma,\\
    B_{\hat{g}} u & = & 0 & \text{on } \partial \Sigma
  \end{array} \right.
\end{equation}
satisfying also $u > 0$ on $\Sigma$ and $u \rightarrow 1$ at infinity.

Following from {\eqref{conformal-rough}}, we have that for any domain $D
\subset \Sigma$, we have that $\lambda_1 (D) > 0$. Using Fredholm alternative
and {\cite[Proposition 3.3]{almaraz-positive-2016}}, we have a unique solution
$v \in C^{2, \alpha}_{\gamma} (\Sigma)$ (see definitions of such weighted
Holder spaces in {\cite{almaraz-positive-2016}}) to
\begin{equation}
  \left\{ \begin{array}{lcll}
    L_{\hat{g}} v & = & - c_{n - 1} R_{\hat{g}}^{\Sigma} & \text{in }
    \Sigma,\\
    B_{\hat{g}} v & = & - c_{n - 1} H_{\hat{g}}^{\partial \Sigma} & \text{on }
    \partial \Sigma .
  \end{array} \right.
\end{equation}
Then $u = v + 1$ is the desired solution.

We turn to positivity of $u$. Suppose now that the set $\Omega = \{x \in
\Sigma : u < 0\}$ is not empty. Since $u \to 1$ at infinity then $\Omega$ must
be a bounded domain of $\Sigma$. On $\partial \Omega$, $u = 0$. Such $u$
restricted to $\Omega$ will give an eigenfunction $u$ with zero eigenvalue.
However $\lambda_1 (\Omega) > 0$. This is a contradiction. So $u \geqslant 0$.
The strict positivity in the interior $\Sigma$ follows by applying the usual
maximum principle on a domain whose boundary is away from $\partial \Sigma$.
The strict positivity on the boundary $\partial \Sigma$ follows from Hopf's
maximum principle on a ball tangent to the boundary at our chosen point. In
conclusion, $u > 0$.

We see that $u$ has the asymptotics $u = 1 + m_0 | \bar{x} |^{3 - n} + O (|
\bar{x} |^{2 - n})$, in particular has finite mass. Note that the dimension of
$\Sigma$ is $n - 1$. Take $\varphi = u$ in {\eqref{conformal-rough}}, we see
that
\begin{equation}
  - 2 \int_{\partial \Sigma} H_{\hat{g}}^{\partial \Sigma} u^2 - \int_{\Sigma}
  R_{\hat{g}}^{\Sigma} u^2 < 2 \int_{\Sigma} | \nabla u|^2 < \frac{1}{c_{n -
  1}} \int_{\Sigma} | \nabla u|^2 . \label{eq:conformal-2nd}
\end{equation}
Let $\bar{g} = u^{\frac{4}{n - 3}} \hat{g}$, then $(\Sigma^{n - 1}, \bar{g})$
is scalar flat with minimal boundary according to {\eqref{neumann}}. We turn
to the mass of $(\Sigma^{n - 1}, \bar{g})$. The mass of $(\Sigma^{n - 1},
\hat{g})$ and $(\Sigma^{n - 1}, \bar{g})$ are denoted respectively $\hat{m}$
and $\bar{m}$. We have
\begin{eqnarray}
  \int_{\Sigma} | \nabla u|^2 & = & \lim_{\sigma \to \infty} \int_{D_{\sigma}}
  | \nabla u|^2 \nonumber\\
  & = & \lim_{\sigma \to \infty} \int_{D_{\sigma}} - u \Delta_{\hat{g}} u +
  \int_{\partial D_{\sigma} \cap \partial \Sigma} u \frac{\partial u}{\partial
  \eta} + \int_{\partial D_{\sigma} \cap \partial \Sigma} u \frac{\partial
  u}{\partial \eta} \nonumber\\
  & = & \lim_{\sigma \to \infty} \int_{D_{\sigma}} - R^{\Sigma}_{\hat{g}} u^2
  - 2 \int_{\partial D_{\sigma} \cap \partial \Sigma} u^2
  H_{\hat{g}}^{\partial \Sigma} + \int_{\partial D_{\sigma} \cap \Sigma} u
  \frac{\partial u}{\partial \eta} . \nonumber
\end{eqnarray}
Considering the decay of $u - 1$ and {\eqref{eq:conformal-2nd}},
\begin{equation}
  \lim_{\sigma \to \infty} \int_{\partial D_{\sigma} \cap \Sigma}
  \frac{\partial u}{\partial \eta} = \lim_{\sigma \to \infty} \int_{\partial
  D_{\sigma} \cap \Sigma} u \frac{\partial u}{\partial \eta} > 0.
\end{equation}
Note that
\begin{equation}
  \bar{m} = \bar{m} - \hat{m} = \lim_{\sigma \to \infty} \int_{\partial
  D_{\sigma} \cap \Sigma} - 4 \frac{\partial u}{\partial \eta} < 0.
\end{equation}
We infer as well that $m_0 < 0$.

In conclusion, $(\Sigma, u^{4 / (n - 3)} \hat{g})$ is asymptotically flat, and
scalar flat with minimal boundary and has negative mass $\bar{m} < 0$. The
contradiction follows inductively from the case $n = 3$. Here we finish the
proof with rigidity statement given by {\cite[Lemma 3.3,
3.4]{almaraz-positive-2016}}.{\qed}

\begin{remark}
  In the case of dimension 3, we can avoid choosing a height $a$ to finish the
  proof. Since in dimension 3, we could utilize a {\itshape{logarithm cutoff
  trick}} on the stability inequality as in {\cite{schoen1979}}. In order to
  be consistent with higher dimensions, we use the {\itshape{strong
  stability}} in every dimension $3 \leqslant n < 8$.
\end{remark}

\section{Geometry of $(\mathbb{T}^{n - 1} \times [0, 1])
\#M_0$}\label{geometry}

In this section, we study the geometry of $(\mathbb{T}^{n - 1} \times [0, 1])
\#M_0$. More specifically, we settle the non-existence of metrics with
positive scalar curvature and minimal boundary, and non-existence of
scalar-flat metrics with mean convex boundary on this manifold $(\mathbb{T}^{n
- 1} \times [0, 1]) \#M_0$. This non-existence result was essentially due to Gromov and Lawson {\cite{gromov-positive-1983}}. For the convenience of the reader, we include our sketch of their proof. Note that their proof used minimal surfaces techniques and we use instead free boundary minimal surfaces.
We then adopt an idea of \
{\cite{lohkamp-scalar-1999}} to modify an asymptotic flat manifold with a
noncompact boundary into a manifold of such form. By keeping track of the
scalar curvature and the mean curvature, we can provide a proof of Theorem
\ref{thm:non-existence}. This proof is simpler in the sense that we are doing
analysis on a compact manifold, and avoiding the analysis of asymptotic
behaviors.

\begin{lemma}
  \label{perturbation-to-strict-inequality1} Given any compact manifold $(M^n,
  g)$ with $R_g \geqslant 0$ in $M$ and $H_g \geqslant 0$ on $\partial M$.
  Then $g$ can be conformally changed to a metric $\tilde{g}$ satisfying
  $R_{\tilde{g}} > 0$ everywhere in $M$ and $H_{\tilde{g}} \equiv 0$ on
  $\partial M$ unless $M$ is Ricci flat with totally geodesic boundary.
\end{lemma}

\begin{proof}
  We have two cases to consider.

  {{\em Case i:\/}} $R_g > 0$ somewhere in $M$ or $H_g > 0$ somewhere on
  $\partial M$. The existence of $\tilde{g}$ follows from Lemma
  \ref{lm:positive minimal} by considering the eigenvalue problem
  \[ \left\{\begin{array}{ll}
       L (M, g) u = \lambda u & \quad \text{in } M,\\
       B (M, g) u = 0 & \quad \text{ on } \partial M.
     \end{array}\right. \]
  {{\em Case ii:\/}} $R_g \equiv 0$ in $M$ and $H_g \equiv 0$ on $\partial M$.
  We consider a family of metrics $g_t = g + t \gamma$ with $\gamma$ to be
  chosen later. Note that $t = 0$, $R_0 \equiv 0$ and $H_0 \equiv 0$. Suppose
  that $\lambda_t$ is the first eigenvalue of $L_t := L_{g_t}$ and $B_t :=
  B_{g_t}$, by variational characterization of eigenvalues
  \[ \lambda_t = \frac{\int_M (| \nabla_t u_t |^2 + c_n R_t u^2) \mathrm{d}
     \mu + 2 c_n \int_{\partial M} H_t u^2_t \mathrm{d} \sigma_t}{\int_M u_t^2
     \mathrm{d} \mu_t} . \]
  Note that $u_0 = 1$ and $\lambda_0 = 0$. We use the dot notation to denote
  differentiation with respect to $t$ and evaluation at $t = 0$. For instance,
  $\dot{\lambda} = \lambda' (0)$. In differentiation, all terms drop out
  except for terms involving $\dot{R}$ and $\dot{H}$, hence
  \begin{eqnarray*}
    \dot{\lambda} & = & \frac{c_n}{\ensuremath{\operatorname{vol}} (M, g)}
    [\int_M \dot{R} \mathrm{d} \mu + 2 \int_{\partial M} \dot{H} \mathrm{d}
    \sigma]\\
    & = & \frac{c_n}{\ensuremath{\operatorname{vol}} (M, g)} [\int_M \langle
    \ensuremath{\operatorname{Ric}}_g - \frac{1}{2} R_g g, \gamma_{} \rangle
    \mathrm{d} \mu + \int_{\partial M} (A_{i j} - H_g g_{i j}) \delta g^{i j}
    \mathrm{d} \sigma]\\
    & = & \frac{c_n}{\ensuremath{\operatorname{vol}} (M, g)} [\int_M \langle
    \ensuremath{\operatorname{Ric}}_g, \gamma \rangle \mathrm{d} \mu +
    \int_{\partial M} A_{i j} \delta g^{i j} \mathrm{d} \sigma] .
  \end{eqnarray*}
  If $M$ is not Ricci flat, we choose $\gamma_{i j} = \varphi R_{i j}$ where
  $\varphi$ is a supported away from the boundary and positive in the
  sufficiently neighborhood around a point where $R_{i j}$ is not zero. This
  makes $\dot{\lambda} > 0$. If $M$ is Ricci flat, but the boundary is not
  totally geodesic i.e. $A_{i j} \neq 0$, choosing $\gamma$ such that $\delta
  g^{i j} = A^{i j}$ on the boundary will lead to $\dot{\lambda} > 0$.
  Therefore, we can do the same thing now as {\itshape{Case i}}.
\end{proof}

A similar argument gives the following lemma,

\begin{lemma}
  \label{mean convex boundary deformation}Assume that $(M^n, g)$ satisfies the
  conditions in Lemma \ref{perturbation-to-strict-inequality1}, then the
  metric $g$ can be conformally changed to a metric with $R_{\tilde{g}} \equiv
  0$ in $M$ and $H_{\tilde{g}} > 0$ everywhere on $\partial M$ unless $M$ is
  Ricci flat with totally geodesic boundary.
\end{lemma}

\begin{proof}
  The proof is similar to the previous lemma, except that we consider the
  Steklov-type eigenvalue problem instead:
  \begin{equation}
    L (M, g) \phi = 0 \quad \text{in } M, \hspace{2.0em} B (M, g) \phi =
    \lambda \phi \quad \text{on } \partial M
  \end{equation}
  We omit the details.
\end{proof}

In the following, we generalize Bochner's theorem {\cite[Chapter 7, Section
3]{petersen-riemannian-1998}} to manifolds with boundary. Denote

\begin{align*}
  H_T & = \left\{ \omega \in \wedge^1 M : d \omega = 0, \delta \omega = 0, \nu
  \lrcorner \omega = 0 \text{ on } \partial M \right\},\\
  H_N & = \{\omega \in \wedge^1 M : d \omega = 0, \delta \omega = 0,
  \nu^{\flat} \wedge \omega = 0 \text{ on } \partial M\} .
\end{align*}

For more details, refer to a good exposition of Hodge-Morrey theory in
{\cite[Chapter 5]{giaquinta-cartesian-1998}}.

\begin{lemma}
  Given any compact manifold $(M^n, g)$ with nonnegative Ricci curvature
  $\ensuremath{\operatorname{Ric}}_g \geqslant 0$ with boundary whose second
  fundamental form is nonnegative, then every harmonic 1-form \ $\omega \in
  H_T \cup H_N$ is parallel.
\end{lemma}

\begin{proof}
  Let $e_i$ be orthonormal frame where $e_1 = \nu$ on $\partial M$ and
  $\theta^i$ be its dual frame. Recall the Bochner-Wietzenbock formula for
  $\omega$,
  \[ 0 = (d \delta + \delta d) \omega = \nabla^2_{e_i, e_i} \omega + \theta^i
     \wedge (e_j \lrcorner R (e_i, e_j) \omega) \]
  where $R (X, Y) = \nabla_X \nabla_Y - \nabla_Y \nabla_X - \nabla_{[X, Y]}$.

  Using integration by parts and direct calculation,

  \begin{align*}
    0 = \langle \Delta \omega, \omega \rangle & = \int_M \langle
    \nabla^2_{e_i, e_i} \omega, \omega \rangle + \langle \omega, \theta^i
    \wedge (e_j \lrcorner R (e_i, e_j) \omega) \rangle\\
    & = \int_{\partial M} \langle \nabla_{\nu} \omega, \omega \rangle -
    \int_M (\ensuremath{\operatorname{Ric}}(\omega, \omega) + | \nabla \omega
    |^2) .
  \end{align*}
  If $\omega \in H_N$, since $\omega$ is a 1-form, we can assume that $\omega
  = \phi \theta^1$ on $\partial M$. Since $0 = \delta \omega = \sum_{i = 1}^n
  \theta^i \lrcorner \nabla_{e_i} \omega =\ensuremath{\operatorname{div}}_M
  \omega$, we have

  \begin{align*}
    \langle \nabla_{\nu} \omega, \omega \rangle & = \phi (\nabla_{\nu} \omega,
    \nu)\\
    & = \phi \ensuremath{\operatorname{div}}_M \omega - \sum_{i = 2}^n \phi
    (\nabla_{e_i} \omega, e_i)\\
    & = - \phi^2 H_{\partial M} = - | \omega |^2 H_{\partial M}
  \end{align*}
  and then
  \[ - \int_{\partial M} H_{\partial M} | \omega |^2 - \int_M
     (\ensuremath{\operatorname{Ric}}(\omega, \omega) + | \nabla \omega |^2) =
     0. \]
  If $\omega \in H_T$, we have $\omega^{\sharp}$ is tangent to $\partial M$,
  extend $\nu$ to all of $M$ such that $\omega^{\sharp}$ is orthogonal to
  $\nu$ in an open neighborhood of $\partial M$. Then along $\partial M$,

  \begin{align*}
    0 = d \omega (\omega^{\sharp}, \nu) & = \omega^{\sharp} (\omega (\nu)) -
    \nu (\omega (\omega^{\sharp})) - \omega (\nabla_{\omega^{\sharp}} \nu -
    \nabla_{\nu} \omega^{\sharp})\\
    & = - \langle \nabla_{\omega^{\sharp}} \nu, \omega^{\sharp} \rangle -
    (\nabla_{\nu} \omega, \omega^{\sharp})\\
    & = - A^{\partial M} (\omega^{\sharp}, \omega^{\sharp}) - \langle
    \nabla_{\nu} \omega, \omega \rangle
  \end{align*}
  will give
  \[ - \int_{\partial M} A^{\partial M} (\omega^{\sharp}, \omega^{\sharp}) -
     \int_M (\ensuremath{\operatorname{Ric}}(\omega, \omega) + | \nabla \omega
     |^2) = 0. \]
  In either case, by nonnegativity of $\ensuremath{\operatorname{Ric}}$ and
  second fundamental form of $\partial M$, it is necessary that $| \nabla
  \omega | = 0$ i.e. $\omega$ is parallel.
\end{proof}

In order to apply these lemmas, we still need that the connected sum
$(\mathbb{T}^{n - 1} \times [0, 1]) \#M_0$ is not Ricci flat with totally
geodesic boundary. We now denote the manifold $(\mathbb{T}^{n - 1} \times [0,
1]) \#M_0$ by $(M^n, g)$. We assume the contrary i.e.
$\ensuremath{\operatorname{Ric}}_g \equiv 0$ and $A_g \equiv 0$. Since the
boundary is totally geodesic, we can glue two copies of $M$ along the boundary
to get a Ricci flat manifold and reduce to the closed case. However, in line
with previous lemmas, we apply again Hodge-Morrey theory with boundary.

\begin{lemma}
  \label{notflat}The manifold $M$ is NOT Ricci flat with totally geodesic
  boundary.
\end{lemma}

\begin{proof}
  Assume the contrary, we look at the degree 1 map from $M$ to $\mathbb{T}^{n
  - 1} \times [0, 1]$:
  \[ \pi : M \to \mathbb{T}^{n - 1} \times [0, 1] . \]
  Given the standard coordinates on $\mathbb{T}^{n - 1} \times [0, 1]$, choose
  the form $d x^i$, we pull it back using $\pi$ to $M$ i.e. $\theta_i =
  \pi^{\ast} d x^i$, we have that $[\theta^i] \in H^1 (M, \mathbb{Z})$ for $2
  \leqslant i \leqslant n$ and $\theta_1$ is a representative element of a
  relative necessary class in $H^1 (M, \partial M, \mathbb{Z})$. Since the
  degree of $\pi$ is 1,
  \[ \int_M \theta_1 \wedge \theta_2 \wedge \cdots \wedge \theta_n = 1. \]
  Using Hodge-Morrey theory {\cite[Chapter 5, Section
  2]{giaquinta-cartesian-1998}}, for $2 \leqslant i \leqslant n$, we can
  modify $\theta_i$ to its harmonic representative $\theta_i^H \in H_T$ and
  $[\theta_i^H] \in H^1 (M, \mathbb{Z})$. Similarly, we can modify $\theta_1$
  to its harmonic representative $\theta_1^H \in H_N$ and $[\theta_1^H] \in
  H^1 (M, \partial M, \mathbb{Z})$. By the previous lemma, these $\theta_i^H$
  are parallel. Let $\theta_i = \theta_i^H + d a_i$ where $a_i$ are functions.
  Let $i : \partial M \to M$ denotes the canonical injection, then along
  $\partial M$, $\nu \wedge \theta_1^H = 0$, $i^{\ast} \theta_1^H = 0$ and
  $i^{\ast} a_1 = a_1 |_{\partial M} = 0$. Since these forms $\{\theta_i \},
  \{\theta_i^H \}$ are closed, the difference
  \[ \int_M \theta_1 \wedge \cdots \wedge \theta_n - \int_M \theta_1^H \wedge
     \cdots \wedge \theta_n^H \]
  by Stokes theorem can be transformed into integrals on the boundary. Every
  one of those integrals contains either $i^{\ast} \theta_1^H$ or $i^{\ast}
  a_1$, therefore vanishes. So
  \[ \int_M \theta_1^H \wedge \cdots \wedge \theta_n^H = 1. \]
  This says that $\{\theta_i^H \}$ are non-trivial and form a parallel basis
  for the cotangent bundle $T^{\ast} M$. So $M$ is flat and this contradiction
  with our initial assumption finishes our proof.
\end{proof}

\subsection{Non-existence results}

We now assume on the contrary that $M = (\mathbb{T}^{n - 1} \times [0, 1])
\#M_0$ admits a metric $g$ with $R_g \geqslant 0$ and $H_g \geqslant 0$. By
previous lemmas, we can further assume that $R_g$ is strictly positive.

Before we get to the proof, recall that a well know duality result in
algebraic topology

\begin{theorem}
  \label{poincare-lefschetz}(Poincare-Lefschetz duality) Let $M$ be a manifold
  with boundary with fundamental class $z \in H_n (M, \partial M)$. Then the
  duality maps
  \begin{equation}
    D : H^k (M) \to H_{n - k} (M, \partial M) \label{lefschetz1}
  \end{equation}
  and
  \begin{equation}
    D : H^k (M, \partial M) \to H_{n - k} (M) \label{lefschetz2}
  \end{equation}
  given by taking cap product with $z$ are both isomorphisms.
\end{theorem}

\subsection{Proof of Theorem \ref{thm:non-existence}}

Now we are ready to finish the proof of Theorem \ref{thm:non-existence}.

\begin{proof}[Proof of Theorem \ref{thm:non-existence}.] There exists a map $\pi
: M \rightarrow \mathbb{T}^{n - 1} \times [0, 1]$ we choose standard
normalized forms $\mathrm{d} x^i$ and its pull back $\theta_i := \pi^{\ast}
\mathrm{d} x^i$ to $M$. The fundamental class $z \in H_n (M, \partial M,
\mathbb{Z})$, then the cap product $[\theta_n] \frown z$ is in $H_{n - 1} (M,
\partial M, \mathbb{Z})$ by Lefschetz duality and nonzero. We have the
following minimization procedure,
\begin{equation}
  | \Sigma | = \min \{ | \Sigma_0 | : \Sigma_0 \in H_{n - 1} (M, \partial M,
  \mathbb{Z}) \}
\end{equation}
and
\begin{equation}
  \int_{\Sigma} \theta_1 \wedge \theta_2 \cdots \wedge \theta_{n - 1} = 1.
\end{equation}
Note that $\Sigma$ is area minimizing, hence stable. When $4 \leqslant n < 7$,
the corresponding Rayleigh quotient, the stability inequality
{\eqref{stability-use}} and Lemma \ref{lm:positive minimal} give a metric
$\hat{g}$ on $\Sigma$ such that $R_{\hat{g}} > 0$ and $H_{\hat{g}} = 0$.

Hence we have a $(n - 1)$ dimensional manifold $(\Sigma^{n - 1}, \hat{g})$
whose scalar curvature is positive and boundary is minimal together with forms
$\theta_1, \cdots, \theta_{n - 1}$.

Inductively doing this, we get down to a compact oriented surface $\Sigma^2$
with at least two boundary components. This will lead to $\chi (\Sigma^2)
\leqslant 0$. However, by taking $\varphi \equiv 1$ in the stability
inequality {\eqref{stability-use}},
\[ 2 \pi \chi (\Sigma^2) = \int_{\Sigma^2} \frac{1}{2} R_{\Sigma^2} +
   \int_{\partial \Sigma^2} H_{\partial \Sigma^2} > 0 \]
gives positive Euler characteristic. This is a contradiction. When $n = 3$, we
get directly a surface $\Sigma^2$ with boundary and we apply stability
{\eqref{stability-use}} by inserting $\varphi \equiv 1$ directly.
\end{proof}

\subsection{Another proof of Theorem \ref{thm:non-existence}}
\begin{proof}[Gromov and Lawson's proof.]
In our last proof, we make use of {\eqref{lefschetz1}} of Poincare-Lefschetz
duality (Theorem \ref{poincare-lefschetz}). Also note that by Lemma \ref{mean
convex boundary deformation}, it is possible to deform the manifold $M$ into a
scalar-flat manifold whose mean curvature of the boundary is strictly
positive. In fact, we can give another proof using the other part of the
duality {\eqref{lefschetz2}}.

As in Lemma \ref{notflat}, we have that $[\theta^i] \in H^1 (M, \mathbb{Z})$
for $2 \leqslant i \leqslant n$ and $\theta_1$ is a representative element of
a relative cohomology class in $H^1 (M, \partial M, \mathbb{Z})$. Then
$[\theta_1] \frown z \in H_{n - 1} (M, \mathbb{Z})$, by the same reasoning,
$[\theta_1] \frown z$ is not trivial. We then have the following minimization
procedure,
\begin{equation}
  | \Sigma | = \min \{ | \Sigma_0 | : \Sigma_0 \in H_{n - 1} (M, \mathbb{Z})
  \}
\end{equation}
and
\begin{equation}
  \int_{\Sigma} \theta_2 \wedge \cdots \wedge \theta_n = 1.
\end{equation}
The existence of a minimizer is due to nontriviality of $H_{n - 1} (M,
\mathbb{Z})$ and that two boundaries of $M$ has strictly positive mean
curvature thus acting as barriers. Then $\Sigma$ has no boundary and does not
intersect the boundary $\partial M$.

{\itshape{Case}} $n = 3$:

By construction, on $\Sigma$ we have two closed 1-forms $\theta_2, \theta_3$
such that $\Sigma$ is dual to $\theta_2 \wedge \theta_3$ and that
\[ \int_{\Sigma} \theta_2 \wedge \theta_3 = 1. \]
This fact leads to that $\Sigma$ has positive genus. Because otherwise
$\Sigma$ was a 2-sphere, the closed 1-forms $\theta_2, \theta_3$ would have to
be exact by cohomology of the 2-sphere. But then by the Stokes theorem,
$\int_{\Sigma} \theta_2 \wedge \theta_3 = 0$.

If $n = 3$ by the stability inequality {\eqref{stability-use}}, we also have
that the Euler characteristic
\begin{equation}
  \chi (\Sigma^2) = \int_{\Sigma} \frac{1}{2} R_{\Sigma} + \int_{\partial
  \Sigma} H_{\partial \Sigma} = \frac{1}{2} \int_{\Sigma} R_{\Sigma} > 0
\end{equation}
gives that $\Sigma^2$ has zero genus. This is a contradiction.

{\itshape{Case}} $4 \leqslant n < 8$:

Since $\Sigma$ has no boundary, we consider the eigenvalue problem
\begin{equation}
  L_{\Sigma} \phi = \sigma \phi \quad \text{in } \Sigma .
\end{equation}
Let the first eigenfunction be $u$. The function $u$ is positive on $\Sigma$.

Using the Rayleigh quotient
\begin{equation}
  \sigma_1 = \frac{\int_{\Sigma} (| \nabla u |^2 + c_{n - 1} R_{\Sigma}
  u^2)}{\int_{\Sigma} u^2}, \label{rayleigh-quotient-conformal-2}
\end{equation}
the stability {\eqref{stability-use}} (without integrals over the boundary)
and that $c_{n - 1} < \frac{1}{2}$, we see that $\sigma_1 > 0$.

Let $\hat{g} = u^{\frac{4}{n - 3}} g_{\Sigma}$ (note here that the dimension
of $\Sigma$ is $n - 1$) with this conformal change we have a closed
$2$-surface $(\Sigma, \hat{g})$, whose scalar curvature is positive with forms
$\theta_2, \cdots, \theta_n$. This is the standard case. See for example the
earlier work of Schoen and Yau on positive scalar curvature
{\cite{schoen-structure-1979}}.

Inductively doing this, we get down to a compact oriented surface $\Sigma^2$
just like when $n = 3$.
\end{proof}
\subsection{Relation with positive mass theorem}

We recall the density theorem {\cite[Proposition 4.1]{almaraz-positive-2016}}:

\begin{proposition}
  Given any asymptotically flat manifold $(M^n, g)$ with a noncompact
  boundary, given $\epsilon > 0$, there exists metric $\tilde{g}$ such that

  1. $(M^n, \tilde{g})$ is asymptotically flat.

  2. $(M^n, \tilde{g})$ satisfies $R_{\tilde{g}} = 0$ and $H_{\tilde{g}} = 0$.

  3. $\tilde{g}$ is conformally flat near infinity i.e. $\tilde{g}$ is of the
  form $u^{\frac{4}{n - 2}} \delta$ near infinity with $\Delta u = 0$ in
  $\mathbb{R}_+^n$ and $\frac{\partial u}{\partial x_1} = 0$ on
  $\mathbb{\partial R}_+^n$ for $| x |$ large.

  4. $| m_g - m_{\tilde{g}} | < \epsilon$.
\end{proposition}

We further modify $\tilde{g}$ as follows

\begin{proposition}
  Suppose that $M$ is given as above, if $m_g < 0$, we can deform $\tilde{g}$
  into $\bar{g}$ such that

  1. $(M, \bar{g})$ is scalar flat with zero mean curvature boundary:
  $R_{\bar{g}} = 0$ and $H_{\bar{g}} = 0$.

  2. $\bar{g}$ is exactly Euclidean outside a compact set.
\end{proposition}

\begin{proof}
  The proof of the case of the standard asymptotically flat manifold as in
  {\cite[Proposition 6.1]{lohkamp-scalar-1999}} carries over. It is sufficient
  to take \ into consideration the fact that the functions $h$ and $f$
  constructed in {\cite[Lemma 6.2]{lohkamp-scalar-1999}} satisfy
  $\frac{\partial h}{\partial x_1} = \frac{\partial f}{\partial x_1} = 0$
  along the boundary.
\end{proof}

\begin{proof}[Lohkamp style proof of the positive mass theorem Theorem \ref{positive mass theorem}.] We assume on the
contrary that $m_g < 0$. By the last two propositions, we modify the metric
$g$ into $\bar{g}$. We take a large $\Lambda > 0$ such that the region
$\left\{ x \in M : \hspace{1em} | x_i | \geqslant \Lambda \right\}$ is
Euclidean, we identify $\{ x_i = \Lambda \}$ and $\{ x_i = - \Lambda \}$ for
all $2 \leqslant i < n$ and then cut off the region outside $\{x_1 >
\Lambda\}$, we obtain a compact manifold $M$ with boundary $\partial M$(with
two components at least) with $R_M \geqslant 0$ and $H_{\partial M} \geqslant
0$, and at some point $R_M > 0$. Then we see that this contradicts the
non-existences results of Theorem \ref{thm:non-existence}. Hence, we have yet
another proof of the positive mass theorem. For the rigidity statement, see
the article {\cite[Lemma 4.3,
4.4]{almaraz-positive-2016}}.
\end{proof}

\appendix\section{Details of Computations}

\subsection{Second variation of minimal
hypersurfaces with free boundary}

\begin{proof}[Proof of Theorem \ref{2nd variation}\ ]Let $F = F (x, t) :
\Sigma \times (- \varepsilon, \varepsilon) \to M$ be a 1-parameter family of
diffeomorphisms of $\Sigma$ induced by $X$. We consider coordinates $x^i$ near
a point $p \in \Sigma$, let
\[ e_i = d F \left( \frac{\partial}{\partial x^i} \right), X = d F \left(
   \frac{\partial}{\partial t} \right) . \]
We can further assume that $x^i$ form a normal coordinate system at $p \in
\Sigma$, hence
\[ g_{i j} (p, 0) = \delta_{i j} \text{ and } \nabla_{e_i} e_j (p, 0) = 0. \]
Now we calculate the variation of $d \tmop{vol}_{\Sigma}$. First, under local
coordinates, we have
\begin{equation}
  \partial_t \sqrt{g} = g^{i j} \langle \nabla_{e_i} X, e_j \rangle \sqrt{g}
  \label{var-vol} .
\end{equation}
We calculate the variation $\partial_t g^{i j}$ and $\partial_t \langle
\nabla_{e_i} X, e_j \rangle$.

\begin{align}
  \partial_t g^{i j} & = - g^{i l} g^{j k} \partial g_{k l}\\
  & = - \partial_t g_{i j}\\
  & = - \langle \nabla_X e_i, e_j \rangle - \langle \nabla_X e_j, e_i
  \rangle\\
  & = - \langle \nabla_{e_i} X, e_j \rangle - \langle \nabla_{e_j} X, e_i
  \rangle \label{var-inverse-metric}
\end{align}
evaluated at $(p, 0)$. And similarly
\begin{align*}
  \partial_t \langle \nabla_{e_i} X, e_j \rangle & = \langle \nabla_X
  \nabla_{e_i} X, e_j \rangle + \langle \nabla_{e_i} X, \nabla_X e_j \rangle\\
  & = \langle R (X, e_i) X, e_j \rangle + \langle \nabla_{e_i} X,
  \nabla_{e_j} X \rangle + \langle \nabla_{e_i} \nabla_X X, e_j \rangle .
\end{align*}
Hence we have, evaluating at $(p, 0)$
\begin{align*}
  F_X \assign \partial_t^2 \sqrt{g} & = - [\langle \nabla_{e_i} X, e_j \rangle
  + \langle \nabla_{e_j} X, e_i \rangle] \langle \nabla_{e_i} X, e_j \rangle\\
  & \hspace{1em} + \langle R (X, e_i) X, e_i \rangle + \langle \nabla_{e_i}
  X, \nabla e_i X \rangle + \langle \nabla_{e_i} \nabla_X X, e_i \rangle\\
  & \hspace{1em} + \langle \nabla_{e_i} X, e_i \rangle \langle \nabla_{e_j}
  X, e_j \rangle\\
  & = \langle R (X, e_i) X, e_i \rangle + \langle \nabla_{e_i} X, \nu \rangle
  \langle \nabla_{e_j} X, \nu \rangle\\
  & \hspace{1em} + (\tmop{div} X)^2 + \tmop{div} Z - \langle \nabla_{e_j} X,
  e_i \rangle \langle \nabla_{e_i} X, e_j \rangle .
\end{align*}
Let $X = T + \varphi \nu$ and $Z = \nabla_X X = \hat{Z} + \phi \nu$, since
$\Sigma$ is minimal, we have that $\tmop{div} (\chi \nu) = 0$ for any function
$\chi$, so $\tmop{div} X = \tmop{div} T$ and $\tmop{div} Z = \tmop{div}
\hat{Z}$. Calculating term by term
\[ \langle R (X, e_i) X, e_i \rangle = \langle R (T, e_i) T, e_i \rangle + 2
   \varphi \langle R (T, e_i) \nu, e_i \rangle - \varphi^2 \tmop{Ric} (\nu) \]
and

\begin{align*}
  \langle \nabla_{e_i} X, \nu \rangle \langle \nabla_{e_i} X, \nu \rangle & =
  \langle \nabla_{e_i} (T + \varphi \nu), \nu \rangle \langle \nabla_{e_i} (T
  + \varphi \nu), \nu \rangle\\
  & = [\langle \nabla_{e_i} T, \nu \rangle + \nabla_{e_i} \varphi] [\langle
  \nabla_{e_i} T, \nu \rangle + \nabla_{e_i} \varphi]\\
  & = (B (T, e_i))^2 - 2 B (T, \nabla \varphi) + | \nabla \varphi |^2
\end{align*}
and
\begin{align*}
  & - \langle \nabla_{e_j} X, e_i \rangle \langle \nabla_{e_i} X, e_j
  \rangle\\
  = & - \langle \nabla_{e_j} (T + \varphi \nu), e_i \rangle \langle
  \nabla_{e_i} (T + \varphi \nu), e_j \rangle\\
  = & - [\langle \nabla_{e_j} T, e_i \rangle + \varphi \langle \nabla_{e_j}
  \nu, e_i \rangle] [\langle \nabla_{e_i} T, e_j \rangle + \varphi \langle
  \nabla_{e_i} \nu, e_j \rangle]\\
  = & - \langle \nabla_{e_j} T, e_i \rangle \langle \nabla_{e_i} T, e_j
  \rangle - \varphi^2 B (e_i, e_j)^2 - 2 \varphi \langle \nabla_{e_i} T, e_j
  \rangle B (e_i, e_j) .
\end{align*}
The density of the second variation can therefore be written
\[ F_X = - \varphi^2 \tmop{Ric} (\nu) - \varphi^2 | B |^2 + | \nabla \varphi
   |^2 + G \]
where
\begin{eqnarray*}
  G & = & [\langle R (T, e_i) T, e_i \rangle + (\tmop{div} X)^2 + B (T, e_i)^2
  - \langle \nabla_{e_i} T, e_j \rangle \langle \nabla_{e_j} T, e_i \rangle]\\
  &  & + \tmop{div} \hat{Z} + [2 \varphi \langle R (T, e_i) \nu, e_i \rangle
  - 2 B (T, \nabla \varphi) - 2 \varphi \langle \nabla_{e_i} T, e_j \rangle B
  (e_i, e_j)]\\
  & = & \tmop{div} (T \tmop{div} T - \nabla_T T) + \tmop{div} \hat{Z} - 2
  (\varphi B_{i j} T_i)_{; j}
\end{eqnarray*}
The last line is the consequence of the following two identities.
\begin{align}
  & \tmop{div} (T \tmop{div} T - \nabla_T T)\\
  = & \langle R (T, e_i) T, e_i \rangle + (\tmop{div} X)^2 + B (T, e_i)^2 -
  \langle \nabla_{e_i} T, e_j \rangle \langle \nabla_{e_j} T, e_i \rangle
  \label{1stdiv}
\end{align}
and
\begin{align}
  & - 2 (\varphi B_{i j} T_i)_{; j}\\
  = & 2 \varphi \langle R (T, e_i) \nu, e_i \rangle - 2 B (T, \nabla \varphi)
  - 2 \varphi \langle \nabla_{e_i} T, e_j \rangle B (e_i, e_j) \label{2div} .
\end{align}
Indeed, we calculate as follows
\begin{align}
  & \tmop{div} (T \tmop{div} T - \nabla_T T)\\
  = & (\tmop{div} T)^2 + \nabla_T \langle \nabla_{e_i} T, e_i \rangle -
  \tmop{div} (\nabla_T T)\\
  = & (\tmop{div} T)^2 + \langle \nabla_T \nabla_{e_i} T, e_i \rangle -
  \langle \nabla_{e_i} \nabla_T T, e_i \rangle + \langle \nabla_{e_i} T,
  \nabla_T e_i \rangle\\
  = & (\tmop{div} T)^2 + \langle R (T, e_i) T, e_i \rangle + \langle
  \nabla_{[T, e_i]} T, e_i \rangle + \langle \nabla_{e_i} T, e_j \rangle
  \langle \nabla_T e_i, e_j \rangle\\
  & + \langle \nabla_{e_i} T, \nu \rangle \langle \nabla_T e_i, \nu \rangle\\
  = & (\tmop{div} T)^2 + \langle R (T, e_i) T, e_i \rangle + \langle \nabla_T
  e_i, e_j \rangle \langle \nabla_{e_j} T, e_i \rangle\\
  & - \langle \nabla_{e_i} T, e_j \rangle \langle \nabla_{e_j} T, e_i \rangle
  + \langle \nabla_{e_i} T, e_j \rangle \langle \nabla_T e_i, e_j \rangle +
  \langle \nabla_{e_i} T, \nu \rangle \langle \nabla_T e_i, \nu \rangle\\
  = & \langle R (T, e_i) T, e_i \rangle + (\tmop{div} X)^2 + B (T, e_i)^2 -
  \langle \nabla_{e_i} T, e_j \rangle \langle \nabla_{e_j} T, e_i \rangle .
  \label{1divproof}
\end{align}
And
\begin{align}
  - 2 (\varphi B_{i j} T_i)_{; j} & = - 2 \varphi_{; j} B_{i j} T_i - 2
  \varphi B_{i j} T_{i ; j} - 2 \varphi B_{i j ; j} T_i\\
  & = - 2 \varphi_{; j} B_{i j} T_i - 2 \varphi B_{i j} T_{i ; j} - 2 \varphi
  R (T, e_j, e_j, \nu) \label{2divproof}
\end{align}
where we use the Gauss-Codazzi equation
\[ B_{i j ; j} = B_{i j ; j} - \nabla_i H_{} = B_{i j ; j} - B_{j j ; i} = R
   (e_i, e_j, e_j, \nu) \]
and the minimality of $\Sigma$.\hspace*{\fill}
\end{proof}
\subsection{Decay estimates of Minimal surface equations}
\begin{lemma}
  \label{decay calculation}Suppose outside a large enough compact set $K$, $f$
  satisfies the minimal surface equation
  \begin{equation}
    \sum_{i j} \left( \delta_{i j} - \frac{f_{, i} f_{, j}}{1 + | \partial f
    |^2} \right) f_{, i j} + \frac{2 (n - 1)}{n - 2} \sqrt{1 + | \partial f
    |^2} \frac{\partial}{\partial \nu_0} \log h = 0
  \end{equation}
  and on the boundary $\partial \Sigma \sim K$,
  \begin{equation}
    \partial_1 f = 0
  \end{equation}
  with the decay
  \[ |f| + |x' | | \partial f| + |x' |^2 | \partial^2 f| = O (|x' |^{-
     \alpha}) \]
  where $0 < \alpha < 1$. The we can improve this decay rate to
  \[ f (x') = a_0 + a_1 \log |x' | + O (|x' |^{- 1 + \varepsilon}) \]
  if $n = 3$; and
  \[ f (x') = f = a_0 + a_1 |x' |^{3 - n} + O (|x' |^{2 - n + \varepsilon}) \]
  if $n \geqslant 4$.
\end{lemma}

\begin{proof}
  $f$ satisfies a simpler Poisson equation outside a compact set
  \[ \Delta f = \frac{f_{, i} f_{, j}}{1 + | \partial f|^2} f_{, i j} -
     \frac{2 (n - 1)}{n - 2} \sqrt{1 + | \partial f |^2}
     \frac{\partial}{\partial \nu_0} \log h = : g. \]
  Using the asymptotics of $h$, we see that $g = O (|x' |^{\max \{- 1 - 3
  \alpha, - n + 1 - \alpha\}})$. We extend $g$ to all of $\mathbb{R}_+^{n -
  1}$. Using {\cite[Lemma A.1]{almaraz-positive-2016}}, we can solve the
  following PDE
  \[ \left\{\begin{array}{ll}
       \Delta w = g & \quad \text{in } \mathbb{R}_+^{n - 1},\\
       \partial_1 w = 0 & \quad \text{on } \partial \mathbb{R}_+^{n - 1}
     \end{array}\right. \]
  where $w$ satisfies the bound
  \begin{equation}
    w = O (|x' |^{\max \{1 - 3 \alpha, 3 - n - \alpha\}+ \varepsilon})
    \label{inhomogeneous decay}
  \end{equation}
  with any $\varepsilon > 0$. The we have that $v := f - w$ satisfies for
  large $|x' |$the following PDE
  \[ \left\{\begin{array}{ll}
       \Delta v = 0 & \quad \text{in } \mathbb{R}_+^{n - 1},\\
       \partial_1 v = 0 & \quad \text{on } \partial \mathbb{R}_+^{n - 1} .
     \end{array}\right. \]
  When $4 \leqslant n < 8$, as already proved in {\cite[Section
  5]{almaraz-positive-2016}}, we see that
  \begin{equation}
    v = a_0 + a_1 |x' |^{3 - n} + O (|x' |^{2 - n}) . \label{harmonic decay}
  \end{equation}
  Combing {\eqref{inhomogeneous decay}} and {\eqref{harmonic decay}}, we
  obtain an improved decay rate for $f$ and this decay rate can be further
  improved to decay rates similar to that of $v$ by {\cite[Lemma
  A.1]{almaraz-positive-2016}}, i.e. for any given $\varepsilon > 0$,
  \[ f = a_0 + a_1 |x' |^{3 - n} + O (|x' |^{2 - n + \varepsilon}) .
      \]
  For the dimension three case, we replace the kernel $\Gamma (x, y) := \Gamma
  (x - y)$ in {\cite{meyers-expansion-1963}} by
  \[ \Gamma (x, y) = \log |x - y| + \log |x - \tilde{y} |, \]
  we then can proceed in the same way as in dimensions $4 \leqslant n < 8$,
  and we arrive the decay
  \[ f = a_0 + a_1 \log |x' | + O (|x' |^{- 1 + \varepsilon}) . \]
  Hence we finish proving the decay estimates for $f$.
\end{proof}

\begin{remark}
  Since by construction $f$ is actually bounded, in dimension 3, $a_1$ has to
  be $0$, i.e.
  \[ f = a_0 + O (|x' |^{- 1 + \varepsilon}) . \]
  Moreover, there is a slightly different case handle by Schoen
  {\cite{schoen-uniqueness-1983}}.
\end{remark}

\end{document}